\documentclass[12pt, a4]{amsart}
\usepackage{amscd, amsmath, amssymb}
\usepackage{fullpage}
\theoremstyle{plain}
\newtheorem{thm}{Theorem}[section]
\newtheorem{lem}[thm]{Lemma}
\newtheorem{pro}[thm]{Proposition}
\theoremstyle{definition}

\newtheorem{ex}[thm]{Example}
\newtheorem{rem}[thm]{Remark}

\numberwithin{equation}{section}
\newcommand{\R}{\mathbb R}
\DeclareMathOperator{\sgn}{sgn}

\begin{document}

\title[BVP on the half-line with functional BCs]{Positive solutions of BVPs on the half-line involving functional BCs}
\author[G. Infante]{Gennaro Infante}
\address{Gennaro Infante, Dipartimento di Matematica e Informatica, Universit\`{a} della
Calabria, 87036 Arcavacata di Rende, Cosenza, Italy}
\email{gennaro.infante@unical.it}
\author[S. Matucci]{Serena Matucci}
\address{Serena Matucci, Department of Mathematics and
Computer Science ``Ulisse Dini'', University of Florence, 50139 Florence, Italy}
\email{serena.matucci@unifi.it}
\subjclass[2010]{Primary 34B40, secondary 34B10, 34B18}
\keywords{Fixed point index, cone, positive global solution, functional boundary condition.}
\begin{abstract}
We study the existence of positive solutions on the half-line of a second order ordinary differential equation subject to functional boundary conditions. Our approach relies on a combination between the fixed point index for operators on compact intervals, a fixed point result for operators on noncompact sets, and some comparison results for principal and nonprincipal solutions of suitable auxiliary linear equations.
\end{abstract}
\maketitle

\section{Introduction}
In this manuscript we discuss the existence of multiple non-negative solutions of the boundary value problem (BVP)
\begin{equation}\label{BVP-HL-intro}
\begin{cases}
(p(t)u'(t))'+f(t, u(t))=0,\ t\geq 0,\\
\alpha u(0)-\beta u'(0)=B[u],\ u(+\infty)=0,
\end{cases}
\end{equation}
where $p, f$ are continuous functions on their domains, $\alpha>0$, $\beta\geq 0$  and $B$ is a suitable functional with \emph{support} in~$[0,R]$.
The functional formulation of the boundary conditions (BCs) covers, \emph{as special cases}, the interesting setting of (not necessarily linear) multi-point and integral BCs; there exists a wide literature on this topic, we refer the reader to the recent paper~\cite{Chris} and references therein.

The approach that we use to solve the BVP~\eqref{BVP-HL-intro},
in the line of the papers \cite{dmm:p,dmm:d,ghz,mm:b,m:mb}, consists in considering two auxiliary BVPs separately, the first one on the compact interval $[0,R]$, where $B$ has support and $f$ is nonnegative, and the second one on the half-line $[R,\infty)$, where $f$ is allowed to change its sign. Unlike the above cited articles, in which the problem of  gluing the solutions is solved with some continuity arguments and an analysis in the phase space, here both the auxiliary problems have the same slope condition in the junction point (namely the condition $u'(R)=0$), which simplifies the arguments.
This kind of decomposition is some sort of an analogue of one employed by Boucherif and Precup~\cite{bp} utilized for equations with nonlocal initial conditions, where the associated nonlinear integral operator is decomposed into two parts, one of Fredholm-type (that takes into account the functional conditions) and another one of Volterra-type.

We make the following assumptions on the terms that occur in \eqref{BVP-HL-intro}.

\begin{itemize}
\item $p:[0,+\infty)\to (0,+\infty)$ is continuous, with
\begin{equation}\label{P}
 P=\int_R^{+\infty} \frac{1}{p(t)} \, dt <+\infty.
\end{equation}
\item $f:[0,+\infty)\times [0,+\infty)\to \R$ is continuous, with $f(t,v)\geq 0$ for $(t,v)\in [0,R]\times [0,+\infty)$, $f(t,0)=0$ for $t \geq 0$.
\item There exist two continuous functions $b_1, b_2: [R, +\infty)\to \R$, with $b_2\geq 0$, and two nondecreasing $C^1$-functions  $F_1, F_2: [0,+\infty)\to [0,+\infty)$, with $F_j(0)=0$, $F_j(v)>0$ for $v>0$, $j=1,2$, such that
     \begin{equation}
     b_1(t) F_1(v)\leq f(t,v)\leq b_2(t) F_2(v), \quad \text{for all } (t,v)\in [R,+\infty)\times [0,+\infty). \label{FF1}
     \end{equation}
\item For $j=1,2$
 \begin{equation}
 \limsup_{v \to 0^+} \frac{F_j(v)}{v}<+\infty. \label{FF2}
 \end{equation}
\item Let $b_1^-, b_1^+$ be the negative and the positive part of $b_1$, i.e. $b_1^-(t)=\max\{0, -b_1(t)\}$, $b_1^+(t)=\max\{0, b_1(t)\}$. Then
    \begin{align}
    &B_1^-=\int_R^{+\infty} b_1^-(t) \, dt <+\infty,\label{b1-}\\
      &\int_R^{+\infty} \frac{1}{p(t)} \int_R^{t} b_1^+(s)\, ds dt=+\infty,\label{b1+}
    \end{align}
\end{itemize}
Notice that \eqref{P}, \eqref{b1+} imply
\[
\int_R^\infty b_1^+(t) \, dt =+\infty,
\]
and, in particular, the function $b_1$ cannot be negative in a neighbourhood of infinity.

As we will show in Section 3 (see also Theorem~\ref{t:main2}), if the condition~\eqref{b1+} is not satisfied, then our approach leads to the existence of a bounded non-negative solution of the differential equation in~\eqref{BVP-HL-intro}, namely a solution of the problem
\begin{equation}\label{BVP-HL-f}
\begin{cases}
(p(t)u'(t))'+f(t, u(t))=0,\ t\geq 0,\\
\alpha u(0)-\beta u'(0)=B[u],\ u\geq 0 \text{ and  bounded on } [0, \infty).
\end{cases}
\end{equation}

The problem of the existence and multiplicity of the solutions for the equation in \eqref{BVP-HL-intro},
which are non-negative in the interval $[0, R]$ and satisfy the functional BCs and the additional assumptions at $u'(R)=0$, is considered in Section~2 and is solved by means of the classical fixed point index for compact maps.
A BVP on $[R, +\infty)$ is examined in Section~3, where we deal with the existence of positive global solutions
 which have zero initial slope and are bounded or tend to zero at infinity. This second problem is solved by using a fixed point
theorem for operators defined in a Frech\'et space, by a Schauder's linearization device, see~\cite[Theorem 1.3]{Furi}, and does not require the explicit form of the fixed point operator, but
only some a-priori bounds. These estimates are obtained using some properties of principal and nonprincipal
solutions of auxiliary second-order linear equations, see~\cite[Chapter 11]{H} and~\cite{DM}. Finally, the existence and multiplicity of solutions for the BVP~\eqref{BVP-HL-intro} and~\eqref{BVP-HL-f} is obtained in Section~4, thanks to the fact that the problem in $[R, +\infty)$ has at least a solution for every initial value $u(R)$ sufficiently small.
An example completes the paper.

\section{An auxiliary BVP on the compact interval $[0,R]$}
In this Section we investigate the existence of multiple positive solutions of the BVP
\begin{equation}\label{Aux-BVP-comp}
\begin{cases}
(p(t)u'(t))'+f(t,u(t))=0,\ t \in[0,R],\\
\alpha u(0)-\beta u'(0)=B[u],\ u'(R)=0.
\end{cases}
\end{equation}
First of all we recall some results regarding the linear BVP
\begin{equation}\label{Aux-BVP-comp-lin}
\begin{cases}
-(p(t)u'(t))'=0,\quad t \in[0,R],\\
\alpha u(0)-\beta u'(0)=0,\ u'(R)=0.
\end{cases}
\end{equation}
It is known, see for example~\cite{kql-blms-06}, that the Green's function $k$ for the BVP~\eqref{Aux-BVP-comp-lin} is given by
$$
 k(t,s):=\frac{1}{\alpha}\begin{cases}
 \frac{\beta}{p(0)}+\alpha \int_{0}^{s}\frac{1}{p(\mu)}\, d\mu, &s\leq t,\\
 \frac{\beta}{p(0)}+\alpha \int_{0}^{t}\frac{1}{p(\mu)}\, d\mu,\ &s\geq t.
\end{cases}
$$
and satisfies the inequality (see~\cite[Lemma 2.1]{kql-blms-06})
$$
c(t)\Phi(s)\leq k(t,s)\leq \Phi(s),\ (t,s)\in [0,R]^2,
$$
where
\begin{equation}\label{phi-c}
\Phi(s):=\frac{\beta}{\alpha p(0)}+ \int_{0}^{s}\frac{1}{p(\mu)}\, d\mu,\ \text{and}\ c(t):=\frac{ \frac{\beta}{p(0)}+\alpha \int_{0}^{t}\frac{1}{p(\mu)}\, d\mu}{ \frac{\beta}{p(0)}+\alpha \int_{0}^{R}\frac{1}{p(\mu)}\, d\mu}.
\end{equation}
Note that the constant function $\dfrac{1}{\alpha}$ solves the BVP
\begin{equation*}
\begin{cases}
-(p(t)u'(t))'=0,\quad t \in[0,R],\\
\alpha u(0)-\beta u'(0)=1,\ u'(R)=0.
\end{cases}
\end{equation*}
We associate to the BVP~\eqref{Aux-BVP-comp} the perturbed Hammerstein integral equation
\begin{equation}\label{phamm}
u(t)=Fu(t)+\frac{H[u]}{\alpha}:=Tu(t),
\end{equation}
where
$$
Fu(t):= \int_0^R k(t,s)f(s,u(s))\,ds.
$$

We seek fixed points of the operator $T$ in a suitable cone of  the space of continuous functions $C[0,R]$, endowed with the usual norm
 $\|w\|:=\max\{|w(t)|,\; t\; \in [0,R]\}$.

We recall that a \emph{cone} $K$ in a Banach space $X$  is a closed
convex set such that $\lambda \, x\in K$ for $x \in K$ and
$\lambda\geq 0$ and $K\cap (-K)=\{0\}$.
In the following Proposition we recall the main properties of the classical fixed point index for compact maps, for more details see~\cite{Amann-rev, guolak}. In what follows the closure and the boundary of subsets of a cone $K$ are understood to be relative to $K$.

\begin{pro}\label{propindex}
Let $X$ be a real Banach space and let $K\subset X$ be a cone. Let $D$ be an open bounded set of $X$ with $0\in D_{K}$ and
$\overline{D}_{K}\ne K$, where $D_{K}=D\cap K$.
Assume that $T:\overline{D}_{K}\to K$ is a compact operator such that
$x\neq Tx$ for $x\in \partial D_{K}$. Then the fixed point index
 $i_{K}(T, D_{K})$ has the following properties:

\begin{itemize}

\item[$(i)$] If there exists $e\in K\setminus \{0\}$
such that $x\neq Tx+\lambda e$ for all $x\in \partial D_{K}$ and all
$\lambda>0$, then $i_{K}(T, D_{K})=0$.

\item[$(iii)$] If $Tx \neq \lambda x$ for all $x\in
\partial D_{K}$ and all $\lambda > 1$, then $i_{K}(T, D_{K})=1$.

\item[(iv)] Let $D^{1}$ be open bounded in $X$ such that
$\overline{D^{1}_{K}}\subset D_{K}$. If $i_{K}(T, D_{K})=1$ and $i_{K}(T,
D_{K}^{1})=0$, then $T$ has a fixed point in $D_{K}\setminus
\overline{D_{K}^{1}}$. The same holds if
$i_{K}(T, D_{K})=0$ and $i_{K}(T, D_{K}^{1})=1$.
\end{itemize}
\end{pro}
The assumptions above allow us to work in the cone
\begin{equation}\label{cone}
K:=\{u\in C[0,R]:\,u\geq 0,\,\, \min_{t\in [a,b]}u(t)\ge c\|u\|\},
\end{equation}
a type of cone firstly used by Krasnosel'ski\u\i{}, see~\cite{krzab}, and D.~Guo, see e.g.~\cite{guolak}.
In~\eqref{cone} $[a,b]$ is a suitable subinterval of $[0,R]$ and $c:=\min_{t\in[a,b]}c(t)$, with $c(t)$ given by~\eqref{phi-c}. We have freedom of choice of $[a,b]$, with the restriction $a>0$ when $\beta=0$. Note also that the constant function equal to $r\geq 0$ (that we denote, with abuse of notation $r$) belongs to $K$, so $K\neq \{0\}$.

Regarding the functional $H$ we assume that
\begin{itemize}
\item$H: K\to [0,+\infty)$ is continuous and maps bounded sets in bounded sets.
\end{itemize}
With these ingredients it is routine to show that $T$ leaves $K$ invariant and is compact.
We make use of the following open bounded set (relative to $K$)
$$
K_{\rho}:=\{u\in K: \|u\|<\rho\}.
$$

We now employ some \emph{local} upper and lower estimates for the functional $H$, in the spirit of ~\cite{gi-tmna, gi-dcds}.
We begin with
a condition which implies that the index is $1$.
\begin{lem} \label{ind11}
Assume that
\begin{enumerate}
\item[$(\mathrm{I}^{1}_{\rho})$]
there exists
 $\rho>0$, such that the following algebraic inequality holds:
\begin{equation}\label{indice12}
 \frac{1}{m} \overline{f}_{\rho} +\frac{1}{\alpha} \overline{H}_{\rho}< \rho,
\end{equation}
where
$$
\overline{f}_{\rho}:=\max_{(t,u)\in [0,R]\times [0,\rho]}f(t,u),\  \overline{H}_{\rho}:=\sup_{u \in \partial K_{\rho}}H[u]\ \ \text{and}\
\frac{1}{m}:=\sup_{t \in[0,R]}\int_{0}^{R} k(t,s)\,ds.
$$
\end{enumerate}
Then $i_{K}(T,K_{\rho})$ is $1$.
\end{lem}

\begin{proof}
Note that if
$u \in \partial K_{\rho}$  then we have $0\le u(t) \le \rho$ for every $t \in [0,R]$.
We prove that $\mu\, u\neq Tu$ for every $\mu \geq 1$ and $u\in \partial K_{\rho}$.
In fact, if this does
not happen, there exist $\mu \geq 1$ and $u\in \partial K_{\rho}$ such that, for every $t \in [0,R]$, we have
\begin{equation*}
\mu\, u(t)= Tu(t)=Fu(t)+\frac{1}{\alpha}H[u].
\end{equation*}
Then we obtain, for $t\in [0,R]$,
\begin{equation}\label{rel}
\mu\, u(t)\leq \int_{0}^{R} k(t,s)\overline{f}_{\rho}ds + \frac{1}{\alpha} \overline{H}_{\rho}\leq  \frac{1}{m} \overline{f}_{\rho} +\frac{1}{\alpha} \overline{H}_{\rho}.
 \end{equation}
Taking the supremum for $t\in [0,R]$ in \eqref{rel}
and using the inequality~\eqref{indice12} we obtain
$\mu \rho <\rho$, a contradiction that proves the result.
\end{proof}
Now we give a  condition which implies that the index is $0$ on the set $K_{\rho}$.
\begin{lem} \label{indice0}
Assume that
\begin{enumerate}
\item[$(\mathrm{I}^{0}_{\rho})$]
there exists
$\rho>0$ such that the following algebraic inequality holds:
\begin{equation}\label{ind 0 2 +}
\frac{1}{M}\underline{f}_{\rho}+\frac{1}{\alpha}\underline{H}_{\rho}[u]>\rho,
\end{equation}
where
$$
\underline{f}_{\rho}:=\min_{(t,u)\in [a,b]\times [c\rho,\rho]}f(t,u),\  \underline{H}_{\rho}:=\inf_{u \in \partial K_{\rho}}H[u]\ \ and\
\frac{1}{M}:=\inf_{t \in[a,b]}\int_{a}^{b} k(t,s)\,ds.
$$
\end{enumerate}
Then $i_{K}(T,K_{\rho})$ is $0$.
\end{lem}

\begin{proof}
Note that the constant function $1$ belongs to $K$. We prove that $u\not=Tu+\lambda 1$ for every $u\in \partial K_{\rho}$ and for every $\lambda \geq 0$. If this is false, there exist $u\in \partial K_{\rho}$ and $\lambda \geq 0$ such that
$u=Tu+\lambda 1$. Then we have, for $t\in[a,b]$,
\begin{multline}\label{Eq0}
  \rho \geq  u(t)=Fu(t)+\frac{1}{\alpha}H[u]+ \lambda 1\geq Fu(t)+\frac{1}{\alpha}H[u]\\ \geq  \int_{a}^{b} k(t,s)f(s,u(s))\,ds +\frac{1}{\alpha}H[u]
   \geq  \int_{a}^{b} k(t,s)\underline{f}_{\rho}\,ds +\frac{1}{\alpha}\underline{H}_{\rho}[u]\geq \frac{1}{M}\underline{f}_{\rho}+\frac{1}{\alpha}\underline{H}_{\rho}[u].
  \end{multline}
Using the inequality~\eqref{ind 0 2 +} in~\eqref{Eq0} we obtain
$\rho >\rho$,
a contradiction that proves the result.
\end{proof}
In view of the Lemmas above, we may state our  result regarding the existence of one or more nontrivial solutions. Here, for brevity, we provide sufficient conditions for the existence of one, two or three solutions. It is possible to obtain more solutions, by adding more conditions of the same type, see for example~\cite{kqljlms}.
\begin{thm}
\label{thmmsol1}
The BVP~\eqref{Aux-BVP-comp} has at least one non-negative solution $u_1$, with $\rho_1< u_1(R)< \rho_2$ if either of the following conditions holds.
\begin{enumerate}
\item[$(S_{1})$] There exist $\rho _{1},\rho _{2}\in (0, +\infty )$ with $\rho
_{1}<\rho _{2}$ such that $(\mathrm{I}_{\rho _{1}}^{0})$ and $(\mathrm{I}_{\rho _{2}}^{1})$ hold.
\item[$(S_{2})$] There exist $\rho _{1},\rho _{2}\in (0, +\infty )$ with $\rho
_{1}<\rho _{2}$ such that $(\mathrm{I}_{\rho _{1}}^{1})$ and $(\mathrm{I}%
_{\rho _{2}}^{0})$ hold.
\end{enumerate}
The BVP~\eqref{Aux-BVP-comp} has at least two non-negative solutions $u_1$ and $u_2$, with $\rho_1< u_1(R)< \rho_2< u_2(R)< \rho_3$, if either of the following conditions holds.
\begin{enumerate}
\item[$(S_{3})$] There exist $\rho _{1},\rho _{2},\rho _{3}\in (0, +\infty )$
with $\rho _{1}<\rho _{2}<\rho _{3}$ such that $(\mathrm{I}_{\rho
_{1}}^{0}),$ $(
\mathrm{I}_{\rho _{2}}^{1})$ $\text{and}\;\;(\mathrm{I}_{\rho _{3}}^{0})$
hold.
\item[$(S_{4})$] There exist $\rho _{1},\rho _{2},\rho _{3}\in (0, +\infty )$
with $\rho _{1}<\rho _{2}<\rho _{3}$ such that $(\mathrm{I}%
_{\rho _{1}}^{1}),\;\;(\mathrm{I}_{\rho _{2}}^{0})$ $\text{and}\;\;(\mathrm{I%
}_{\rho _{3}}^{1})$ hold.
\end{enumerate}
The BVP~\eqref{Aux-BVP-comp} has at least three non-negative solutions $u_1$, $u_2$ and $u_3$, with $\rho_1< u_1(R)< \rho_2< u_2(R)< \rho_3< u_3(R)< \rho_4$ if either of the following conditions holds.
\begin{enumerate}
\item[$(S_{5})$] There exist $\rho _{1},\rho _{2},\rho _{3},\rho _{4}\in
(0, +\infty )$ with $\rho _{1}<\rho _{2}<\rho _{3}<\rho
_{4}$ such that $(\mathrm{I}_{\rho _{1}}^{0}),$ $(\mathrm{I}_{\rho _{2}}^{1}),$ $(\mathrm{I}%
_{\rho _{3}}^{0})$ and $(\mathrm{I}_{\rho _{4}}^{1})$ hold.
\item[$(S_{6})$] There exist $\rho _{1},\rho _{2},\rho _{3},\rho _{4}\in
(0, +\infty )$ with $\rho _{1}<\rho _{2}<\rho _{2}<\rho _{3}<\rho _{4}$
such that $(\mathrm{I}_{\rho _{1}}^{1}),$ $(\mathrm{I}_{\rho
_{2}}^{0}),$ $(\mathrm{I}_{\rho _{3}}^{1})$ and $(\mathrm{I}%
_{\rho _{4}}^{0})$ hold.
\end{enumerate}
\end{thm}
\begin{proof}
Assume condition $(S_{1})$ holds, then, by Lemma~\ref{indice0}, we have $i_{K}(T,K_{\rho_1})=0$ and, by Lemma~\ref{ind11}, $i_{K}(T,K_{\rho_2})=1$.
By Proposition~\ref{propindex} we obtain a solution $u_1$ for the integral equation~\eqref{phamm}
in $K_{\rho_2}\setminus \overline{K}_{\rho_1}$. Furthermore note that
$$
 \frac{\partial}{\partial t}k(t,s):=\frac{1}{\alpha}
 \begin{cases}
0,\ &s< t,\\
\frac{1}{p(t)},\ &s> t,
\end{cases}
$$
and therefore $u_1'\geq 0$ in $[0,R]$, which, in turn, implies $u_1(R)=\|u_1\|$.

Assume now that condition $(S_{3})$ holds, then we obtain in addition that $i_{K}(T,K_{\rho_3})=0$. By Proposition~\ref{propindex} we obtain the existence of a second solution $u_2$ of the integral equation~\eqref{phamm} in $K_{\rho_3}\setminus \overline{K}_{\rho_2}$. A similar argument as above holds for the monotonicity of~$u_2$.

The remaining cases are dealt with in a similar way.
\end{proof}

\section{An auxiliary BVP on  the half-line $[R, \infty]$}

In this section we state sufficient conditions for the existence of solutions of the following~BVP
\begin{align}
&(p(t)u'(t))'+f(t, u(t))=0,\ t \geq R, \label{Eq2}\\
&u'(R)=0, \, u(R)=u_0, \, u(t)>0, \, \lim_{t \to \infty} u(t)=0,\label{C:Eq2}
\end{align}
where $u_0>0$ is a given constant.
\medskip
The BVP \eqref{Eq2}, \eqref{C:Eq2} involves both initial and asymptotic conditions, and a global condition (i.e., the positivity on the whole half-line). The continuability at infinity of solutions of (\ref{Eq2})
is not a simple problem, see for example~\cite{Butler}. For instance, the Emden--Fowler
equation
\begin{equation}
\bigl(p(t)u^{\prime}(t)\bigr)^{\prime}+g(t)|u(t)|^{\beta}\sgn (u(t))=0,\label{QL}%
\end{equation}
if $\beta>1$ and $g$ is allowed to take negative
values, has solutions which tend to infinity in finite time, see
\cite{Butler, BG}. Moreover, again in the superlinear case $\beta>1$, if $g$ is non-negative with isolated
zeros, then (\ref{QL}) may have solutions which change sign infinitely many times
in the left neighborhood of some $\bar{t}>R$, and so
these solutions are not continuable to infinity, see \cite{Coff}.
Further, even if global solutions exist (that is, solutions which are defined in the whole half-line $[R, +\infty)$), their positivity is not guaranteed in general. Indeed, \eqref{Eq2} may exhibit the coexistence of nonoscillatory and oscillatory solutions; further, nonoscillatory solutions may have an arbitrary large number of zeros.

The problem \eqref{Eq2}, \eqref{C:Eq2} has been consider in \cite{DM} for nonlinear equations with  $p$-laplacian  operator and nonlinear term $f(t,u(t))=b(t)F(u(t))$. We address the reader to such a paper  for a complete discussion on the issues related to the BVP \eqref{Eq2}, \eqref{C:Eq2} and for a review of the existing literature on related problems. The approach used in \cite{DM} to solve the BVP was based on a combination of the Schauder's (half)-linearization device, a fixed point result in the Fr\'echet space of continuous functions on $[R, +\infty)$, and comparison results for principal and nonprincipal solutions of suitable auxiliary half-linear equations, which allow to find good upper and lower bounds for the solutions of the (half)-linearized problem. The same approach, with minor modifications, allow us to treat also the present case with a general nonlinearity $f(t, u(t))$, under the assumptions \eqref{FF1}, \eqref{FF2}. In the following Proposition we recall the fixed-point result \cite[Theorem 1]{DM}, based on \cite[Theorem 1.3]{Furi}, in the form suitable for the present problem.
\begin{pro}\label{tcfm}
Let $J=[t_{0},\infty).$ Consider
the BVP
\begin{equation}\label{bvpo}
\begin{cases}
(p(t)u^{\prime})^{\prime}+f(t,u)=0, &t\in J,\\
u\in S,
\end{cases}
\end{equation}
where $f$ is a continuous function on $J\times\mathbb{R}$ and $S$ is a subset
of $C^{1}(J,\mathbb{R})$. Let $g$ be a continuous function on $J\times
\mathbb{R}^{2}$ such that
\[
g(t,c,c)=f(t,c)\quad\text{for all }(t,c)\in J\times\mathbb{R},
\]
and assume that there exist a closed convex subset $\Omega$ of $C^{1}(J,\mathbb{R})$ and a bounded closed subset $S_{1}$ of $S\cap\Omega$ which
make the problem
\begin{equation}
\label{bvpoh}
\begin{cases}
(p(t)y^{\prime})^{\prime}+g(t,y,q)=0, & t\in J,\\
y\in S_{1}
\end{cases}
\end{equation}
uniquely solvable for all $q\in\Omega$. Then the BVP (\ref{bvpo}) has at least
one solution in $\Omega$.
\end{pro}
In view of this result, no topological properties of the fixed-point operator are needed to
be checked, since they are a direct consequence of \textit{a-priori} bounds for the solutions of the ``linearized'' problem \eqref{bvpoh}.

We state here the existence results in the form which will be used in the next section, addressing to \cite[Theorem 2]{DM} for the general result,  in case of a factored nonlinearity. For reader's convenience we provide a short proof, focusing only on those points which require some adjustments due to the more general nonlinearity. We point out that the present results are obtained  by using the Euler equation
\begin{equation}\label{e:E}
    t^2 y''+nt y' +\left(\frac{n-1}{2}\right)^2 y=0, \quad n>1,
\end{equation}
as a Sturm majorant of the linearized auxiliary equation (see also \cite[Corollary 3]{DM}), but any other linear equation having a positive decreasing solution can be used as a majorant equation, obtaining different conditions. The first result states sufficient conditions for the existence of a global positive solution of \eqref{Eq2}, bounded on $[R, +\infty)$.

\begin{thm}\label{thmmsup} Assume that \eqref{P}
--\eqref{b1-} hold,  and let
$$\displaystyle{M_j(d)=\sup_{v \in(0, d]} \frac{F_j(v)}{v}}, \quad j=1,2$$ where
$d>0$ is a fixed constant. If
\begin{equation}\label{cG-}
B_1^-\leq \frac{\log 2}{M_1(d) P},
\end{equation}
 and
\begin{equation}\label{cconf1}
p(t)\geq t^n, \quad b_2(t)\leq \frac{(n-1)^2}{4 M_2(d) } \, t^{n-2}, \quad \text{ for all } t \geq R,
\end{equation}
for some $n>1$,  then for every $u_0\in (0, d/2]$ the equation \eqref{Eq2}  has a solution $u$, satisfying $u(R)=u_0, \, u'(R)=0, \, 0<u(t)\leq 2 u_0$ for $t \geq R$.
\end{thm}
\begin{proof} The result follows from \cite[Theorem 2 and Corollary 3]{DM}, with some technical adjustments due to the actual general form of the nonlinearity. Indeed, it is sufficient to observe that, for every continuous function $q:[R, \infty) \to (0, d]$ fixed, the equations
\begin{equation}\label{eml}
\left(  p(t)w^{\prime}(t)\right)^{\prime}-M_1(d) b_1^-(t)\, w(t)=0
\end{equation}
is a Sturm minorant of the linearized equation
\begin{equation}\label{ell}
\left(  p(t)u^{\prime}(t)\right)^{\prime}+\dfrac{f(t, q(t))}{q(t)}\, u(t)=0,
\end{equation}
and \eqref{e:E} is a Sturm majorant, due to \eqref{cconf1} (see for instance \cite{H}). Since \eqref{e:E} is nonoscillatory, and has the  solution $y(t)=t^{-(n-1)/2}$, which is positive decreasing on $[R, +\infty)$, from \cite[Lemma 3]{DM} the solution of \eqref{ell} satisfying the initial conditions $u(R)=u_0$, $u'(R)=0$ exists and is positive on $[R, +\infty)$, since it satisfies $u(t)\geq w_0(t)$ for all $t \geq R$, where $w_0$ is the principal solution of \eqref{eml}. Further, double integration of \eqref{ell} on $[R, t]$, $t>R$, gives
\[
U(t) \leq u_0+M_1(d) B_1^- \int_R^t \frac{U(s)}{p(s)} \, ds, \quad \text{ where }U(t)=\max_{s \in [R,t]} u(s),
\]
and the Gronwall lemma, together with \eqref{cG-}, gives the upper bound $u(t)\leq U(t)\leq 2u_0$. Thus, put
\begin{align*}
S&=\{q\in C^{1}[t_{0},\infty):\,q(R)=u_0,\ q^{\prime}(R)=0,\ q(t)>0\ \text{for } t\geq R\},
\\
\Omega&=\left\{  q\in C^{1}[1,\infty):\, q(R)=u_0,\,q^{\prime}(R)=0, \, w_{0}(t)\leq q(t)\leq2u_0\right\}.
\end{align*}
We have $S\cap\Omega=\Omega$ and the set $S_{1}=\overline{\mathcal{T}(\Omega)}$, where $\mathcal{T}$ is the operator which maps every $q \in \Omega$ into the unique solution of \eqref{ell} satisfying the initial conditions $u(R)=u_0$, $u'(R)=0$, satisfies $S_{1}\subset S\cap\Omega=\Omega$ and is bounded in $C^1[R, +\infty)$ (for the detailed proof see \cite[Theorem 2]{DM}). Then Proposition \ref{tcfm} can be applied, and the existence of a solution of \eqref{Eq2} in the set $S\cap \Omega$ is proved.
\end{proof}

\begin{rem} Clearly, the result in Theorem \ref{thmmsup} holds also if we allow a different upper bound for the solution. More precisely, if we look for a solution satisfying $0<u(t)\leq k \,u_0$ with $k>1$, then it is sufficient to put  $\log k$ instead of $\log 2$ in \eqref{cG-}, and we get the existence of global positive solutions of the Cauchy problem associated with \eqref{Eq2}, for every $u_0$ such that $0<k u_0\leq d$.
\end{rem}

\begin{rem} Since $M_2(d)$ is nondecreasing, condition \eqref{cG-} can be seen as an upper bound for the values of $u_0$ for which \eqref{Eq2} has a global bounded solution. For instance, if $F_1(v)=v^\beta$, $\beta>1$, then $M_1(d)=d^{\beta-1}$ and \eqref{cG-} can be read as
\[
2u_0\leq \left(\frac{\log 2}{P B_1^-}\right)^{\frac{1}{\beta-1}} \quad \text{ if } B_1^-\neq 0,
\]
while, if $F_1(v)=v$, then $M_1(d)=1$ for all $d>0$, and \eqref{Eq2}, \eqref{C:Eq2} has solution for all $u_0>0$, provided \eqref{cconf1} is satisfied.
\end{rem}

If in the statement of Theorem \ref{thmmsup}  we assume in addition the condition \eqref{b1+}, we get the existence of a solution of the BVP \eqref{Eq2}, \eqref{C:Eq2}.
\begin{thm}\label{at:2}
Assume that \eqref{P}--\eqref{b1+} hold, and that $d>0$ exists such that
\eqref{cG-}
 and
\eqref{cconf1}
are satisfied for some $n>1$.  Then, for every $u_0\in (0, d/2]$, the BVP \eqref{Eq2}, \eqref{C:Eq2}  has at least a solution $u$, satisfying
\[0<u(t)\leq 2u_0\text{ for }t\in [R,\infty),\ u^{\prime}(t)<0\text{ for
large }t.\]
\end{thm}
\begin{proof} The proof is analogous to the second part of the proof of \cite[Theorem 2]{DM}, with obvious modifications due to the more general form of the  nonlinear term here considered. In particular notice that \eqref{Eq2} with conditions \eqref{FF1} gives the inequality
\[
\frac{(p(t) u'(t))'}{F_1(u(t))}+b_1(t) \leq 0
\]
for every  positive solution $u$ of \eqref{Eq2} and all $t\geq R$. Thus the arguments in the proof of \cite[Theorem 2]{DM} apply also to the present case.
\end{proof}

We conclude this Section pointing out that the case $b_1(t) \geq 0$ for $t \geq R$ is included in the previous results, and in this case Theorems \ref{thmmsup}, \ref{at:2} have a more simple statement. Indeed, $B_1^-=0$, and therefore \eqref{b1-} and \eqref{cG-} are trivially satisfied. Further, every solution of \eqref{Eq2} is nonincreasing on the whole half-line.

\section{The main result}
Combining  Theorems \ref{thmmsol1} and \ref{thmmsup} or \ref{thmmsol1} and \ref{at:2}, we obtain an existence result for one or more solutions of the BVP \eqref{Eq2} and \eqref{C:Eq2}, respectively. We limit ourself to state results for the existence of one or two solutions, for sake of simplicity. Clearly, as pointed out in Section 2, adding more conditions, with similar arguments it is possible to obtain sufficient conditions for the existence of  three solutions (see Theorem \ref{thmmsol1}) or more solutions.

\begin{thm}\label{t:main2}
Suppose that \eqref{P}--\eqref{b1-} are satisfied, and that there exist  $\rho _{1},\rho _{2}\in (0, +\infty )$, with $\rho
_{1}<\rho _{2}$, such that either ($S_1$) or ($S_2$) holds. If $n>1$ exists, such that \eqref{cG-}, \eqref{cconf1} are satisfied with $d=2\rho_2$,
 then the BVP \eqref{BVP-HL-f} has at least one solution $u_1$, with $u_1(t)\leq 2\rho_2$ for all $t \geq 0$.

If, in addition, there exists $\rho _{3}\in (0, +\infty )$, with $\rho _{2}<\rho _{3}$, such that either ($S_3$) or ($S_4$)
 holds, and \eqref{cG-}, \eqref{cconf1} are satisfied with $d=2\rho_3$,
then the BVP~\eqref{BVP-HL-f} has an additional solution $u_2$, with $u_2(t) \leq 2\rho_3$.
\end{thm}

Notice that, from Theorem \ref{thmmsol1}, the solutions $u_1, u_2$ satisfy $\rho_1<u_1(R)<\rho_2<u_2(R)<\rho_3$ and therefore they are distinct solutions. Further, since solutions on $[0,R]$ are increasing, then $u_1(t)<\rho_2, \, u_2(t)< \rho_3,$ for all $t \in [0,R]$.

\medskip
In case also assumption \eqref{b1+} is satisfied, from the above Theorem we obtain sufficient conditions for the existence of solutions of the BVP \eqref{BVP-HL-intro}.

\begin{thm}\label{t:main1}
Suppose that \eqref{P}--\eqref{b1+} are satisfied, and that there exist  $\rho _{1},\rho _{2}\in (0, +\infty )$, with $\rho
_{1}<\rho _{2}$, such that either ($S_1$) or ($S_2$) holds. If $n>1$ exists, such that \eqref{cG-}, \eqref{cconf1} are satisfied with $d=2\rho_2$,
 then the BVP \eqref{BVP-HL-intro} has at least one solution $u_1$, with $u_1(t)\leq 2\rho_2$ for all $t \geq 0$.

If, in addition, there exists $\rho _{3}\in (0, +\infty )$, with $\rho _{2}<\rho _{3}$, such that either ($S_3$) or ($S_4$)
 holds, and \eqref{cG-}, \eqref{cconf1} are satisfied with $d=2\rho_3$,
then the BVP~\eqref{BVP-HL-intro} has an additional solution $u_2$, with $u_2(t) \leq 2\rho_3$.
\end{thm}

We conclude this Section with the following illustration of our results.

\begin{ex}
Let us consider the BVP
\begin{equation}\label{BVP-ex}
\begin{cases}
(p(t)u'(t))'+b(t)u^2(t)=0,\ t\geq 0,\\
u(0)=B[u],\ u(+\infty)=0,
\end{cases}
\end{equation}
where
$$
p(t)=
\begin{cases}
1, &0\leq t\leq 1,\\
t^{n},& t >1,
\end{cases}
\quad
b(t)=
\begin{cases}
1, &0\leq t\leq 1,\\
\sin^+(\frac{\pi}{2}t)- \dfrac{\mu}{t^2}\sin^-(\frac{\pi}{2}t), & t> 1,
\end{cases}
$$
for some $n>1$ and $\mu>0$, and
\[
B[u]=\frac{1}{2}\sqrt{\int_0^1u(t)\, dt}.\]
The definition of the operator $B$  leads to the natural choice $[0,R]=[0,1]$. By direct computation
one has $m=2$. The choice of  $[a,b]=[1/2,1]$ leads to $c=1/2$ and $M=4$. Furthermore note that $\underline{f}_{\rho}=\frac{1}{4}\rho^2$, $\overline{f}_{\rho}=\rho^2$ and
$$
\frac{1}{2}\sqrt{\rho}\geq B[u]\geq \frac{1}{2}\sqrt{\int_{1/2}^1 u(t)\, dt}\geq \frac{1}{4}\sqrt{\rho},\ \text{for every}\  u\in \partial K_\rho.
$$
Observe that the inequalities
\begin{align*}
 \frac{1}{4}\underline{f}_{\rho_1}+\underline{H}_{\rho_1}[u]&\geq  \frac{1}{4}\bigl(\frac{1}{4}{\rho_1}^2+\sqrt{\rho_1}\bigr) >\rho_1,\\
 \frac{1}{2} \overline{f}_{\rho_2} + \overline{H}_{\rho_2}&\leq  \frac{1}{2} \bigl({\rho_2}^2 +\sqrt{\rho_2}\bigr) < \rho_2,\\
 \frac{1}{4}\underline{f}_{\rho_3}+\underline{H}_{\rho_3}[u]&\geq  \frac{1}{4}\bigl(\frac{1}{4}{\rho_3}^2+\sqrt{\rho_3}\bigr) >\rho_3.
\end{align*}
are satisfied for $\rho_1=1/20$, $\rho_2=3/4$ and $\rho_3=15$. Thus ($S_1$) and ($S_3$) hold for $\rho_1=1/20$, $\rho_2=3/4$ and $\rho_3=15$. Note that \eqref{P}--\eqref{b1-} are satisfied, with $P=1/(n-1)$, $F_1(v)=F_2(v)=v^2$, $b_1(t)=b(t), \, b_2(t)=1$, and $B_1^-< \mu$ holds. Straightforward calculations show that also \eqref{b1+} is satisfied. Since $M_1(d)=M_2(d)=d$, applying Theorem \ref{t:main1} we get the following result:
\begin{itemize}
  \item If $(n-1)^2\geq 6$ then \eqref{BVP-ex} has at least a solution $u_1$, with $0 \leq u_1(t)\leq 3/2$, for every $\mu \leq 2(n-1) \log 2/3$.
  \item If $(n-1)^2\geq 120$ then \eqref{BVP-ex} has at least two distinct solution $u_1, \, u_2$, with $0 \leq u_1(t)\leq 3/2$, $0 \leq u_2(t)\leq 30$, for every $\mu \leq (n-1) \log 2/30$.
\end{itemize}

\end{ex}
\section*{Acknowledgments}
G. Infante and S. Matucci were partially supported by G.N.A.M.P.A. - INdAM (Italy).


\begin{thebibliography}{00}

\bibitem{Amann-rev} H. Amann,
\textit{Fixed point equations and nonlinear eigenvalue problems in ordered Banach spaces},
SIAM. Rev. \textbf{18} (1976), 620--709.

\bibitem{bp} A. Boucherif and R. Precup, \textit{On the nonlocal initial
value problem for first order differential equations}, Fixed Point Theory
\textbf{4} (2003), 205--212.

\bibitem{BG} T. A. Burton and R. C. Grimmer, \textit{On the continuability of the
second order differential equation}, Proc. Amer. Math. Soc. \textbf{29} (1971), 277--283.

\bibitem{Butler} G. J. Butler, \textit{The existence of continuable solutions
of a second order order differential equation, } Can. J. Math. \textbf{XXIX} (1977), 472--479.

\bibitem{Furi} M.~Cecchi, M.~Furi and M.~Marini, \textit{On continuity and compactness of some nonlinear operators associated with differential equations in noncompact intervals}, Nonlinear Anal. \textbf{9} (1985), 171--180.

\bibitem{Coff} C. V. Coffman and D. F.  Ullrich, \textit{On the continuation of
solutions of a certain nonlinear differential equation,} Monatsh. Math. \textbf{71}
(1967), 385--392.

\bibitem{dmm:p} Z. Do\v{s}l\'a, M. Marini and S. Matucci, \textit{Positive solutions of nonlocal continuous second order BVP's}, Dynam. Systems Appl. \textbf{23} (2014),  431--446.

\bibitem{dmm:d} Z. Do\v{s}l\'a, M. Marini and S. Matucci, \textit{A Dirichlet problem on the half-line for nonlinear equations with indefinite weight}, Ann. Mat. Pura Appl. (4) \textbf{196} (2017),  51--64.

\bibitem{DM} Z. Do\v{s}l\'a and S. Matucci, \textit{Ground state solutions to nonlinear equations with p-Laplacian}, Nonlinear Anal. \textbf{184} (2019), 1--16.

\bibitem{ghz} M. Gaudenzi, P. Habets and F. Zanolin, \textit{An example of a superlinear problem with multiple positive solutions}, Atti Sem. Mat. Fis. Univ. Modena \textbf{51} (2003), 259--272.

\bibitem{Chris} C. S. Goodrich,  \textit{Pointwise conditions for perturbed Hammerstein integral equations with monotone nonlinear, nonlocal elements}, Banach J. Math. Anal. \textbf{14}, (2020), 290--312.

\bibitem{guolak} D. Guo and V. Lakshmikantham,
\textit{Nonlinear problems in abstract cones}, Academic Press, Boston,
(1988).

\bibitem{gi-tmna}
G. Infante, \textit{Nonzero positive solutions of a multi-parameter elliptic system with functional BCs}, Topol. Methods Nonlinear Anal. \textbf{52} (2018), 665--675.

\bibitem{gi-dcds}
G. Infante, \textit{Positive and increasing solutions of perturbed Hammerstein integral equations with derivative dependence},
Discrete Continuous Dyn. Syst. Ser. B. \textbf{25} (2020), 691--699.

\bibitem{H} P.~Hartman, \textit{Ordinary Differential Equations}, 2 Ed., Birk\"{a}user, Boston-Basel-Stuttgart, 1982.

\bibitem{krzab} M. A. Krasnosel'ski\u\i{} and P. P. Zabre\u{\i}ko,
{\it Geometrical methods of nonlinear analysis}, Springer-Verlag,
Berlin, (1984).

\bibitem{kqljlms} K. Q. Lan,
\textit{Multiple positive solutions of semilinear differential equations
with singularities}, J. London Math. Soc. \textbf{63}
(2001), 690--704.

\bibitem{kql-blms-06}
K. Q. Lan,
\textit{Multiple positive solutions of semi-positone Sturm-Liouville boundary value problems},
Bull. London Math. Soc. \textbf{38} (2006), 283--293.

\bibitem{mm:b} M. Marini and S. Matucci, \textit{A boundary value problem on the half-line for superlinear differential equations
with changing sign weight}, Rend. Istit. Mat. Univ. Trieste \textbf{44} (2012), 117--132.

\bibitem{m:mb} S. Matucci, \textit{A new approach for solving nonlinear BVP's on the half-line for second order equations and
applications}, Math. Bohem. \textbf{140} (2015), 153--169.

\end{thebibliography}
\end{document}